\documentclass[preprint,12pt,3p]{elsarticle}

\usepackage{amsmath, amssymb, amsfonts, amsbsy, amsthm, latexsym, epsfig, txfonts, palatino, pifont, color, tikz, graphicx, framed}
\usetikzlibrary {positioning}
\definecolor {processblue}{cmyk}{0.96,0,0,0}

\theoremstyle{definition}
\newtheorem{theorem}{Theorem}

\newtheorem{lemma}{Lemma}

\newtheorem{cor}{Corollary}
\newtheorem{rk}{Remark}
\newtheorem{df}{Definition}
\newtheorem{prop}{Proposition}

\newcommand{\field}[1]{\mathbb{#1}}
\newcommand{\G}{\field{G}}

\newcommand{\C}{\field{C}}
\newcommand{\W}{\field{W}}

\journal{Linear Algebra and its Applications}

\makeatletter
\g@addto@macro\@floatboxreset\centering
\makeatother

\begin{document}

\begin{frontmatter}

\title{Generalized Eigenvectors of Isospectral Transformations, Spectral Equivalence and Reconstruction of Original Networks}

\author{Leonid Bunimovich}
\ead{bunimovh@math.gatech.edu}

\author{Longmei Shu\corref{cor1}}
\address{School of Mathematics,
Georgia Institute of Technology,
Atlanta, GA 30332-0160 USA}

\cortext[cor1]{Corresponding author}

\ead{lshu6@math.gatech.edu}

\begin{abstract}
Isospectral transformations (IT) of matrices and networks allow for compression of either object while keeping all the information about their eigenvalues and eigenvectors. We analyze here what happens to generalized eigenvectors under isospectral transformations and to what extent the initial network can be reconstructed from its compressed image under IT. We also generalize and essentially simplify the proof that eigenvectors  are invariant under isospectral transformations and generalize and clarify the notion of spectral equivalence of networks.
\end{abstract}

\begin{mscc}
05C50 \sep 15A18
\end{mscc}


\end{frontmatter}

\section{Introduction}
\label{sec1}

The recently developed theory of Isospectral Transformations (IT) of matrices and networks allowed for advances in various areas and led to several surprising results \cite{webb14}. The effectiveness of these applications raises a natural question regarding the possible limits of this approach. Although the theory of isospectral transformations was initially aimed at reduction (i.e. simplification) of networks while keeping all the information about the spectrum of their weighted adjacency, Laplace, or other matrices generated by a network, it turned out \cite{torres14} that all the information about the eigenvectors of these matrices also gets preserved under ITs.

Therefore it is natural to ask what network information may not be preserved after isospectral compression. The main goal of the present paper is to answer this question. It is shown that generalized eigenvectors typically are not preserved under ITs. We also establish some sufficient conditions under which the information about generalized eigenvectors is preserved under ITs. Some new properties of ITs are found, regarding classes of spectrally equivalent matrices and networks. Particularly it is demonstrated that there are essential differences between the standard notion of isospectral matrices and spectral equivalence of networks. A new proof of the preservation of eigenvectors under ITs is given which is shorter and applicable to a more general situation than the one in \cite{torres14}.

\section{Isospectral Graph Reductions}\label{graph-red}

In this section we recall definitions of the isospectral transformations of graphs and networks.

Let $\W$ be the set of rational functions of the form $w(\lambda)=p(\lambda)/q(\lambda)$, where $p(\lambda),q(\lambda)\in\C[\lambda]$ are polynomials having no common linear factors, i.e., no common roots, and where $q(\lambda)$ is not identically zero. $\W$ is a field under addition and multiplication \cite{webb14}.

Let $\G$ be the class of all weighted directed graphs with edge weights in $\W$. More precisely, a graph $G\in\G$ is an ordered triple $G=(V,E,w)$ where $V=\{1,2,\dots,n\}$ is the \emph{vertex set}, $E\subset V\times V$ is the set of \emph{directed edges}, and $w:E\to\W$ is the \emph{weight function}. Denote by $M_G=(w(i,j))_{i,j\in V}$ the \emph{weighted adjacency matrix} of $G$, with the convention that $w(i,j)=0$ whenever $(i,j)\not\in E$. We will alternatively refer to graphs as networks because weighted adjacency matrices define all static (i.e. non evolving) real world networks.

Observe that the entries of $M_G$ are rational functions. Let's write $M_G(\lambda)$ instead of $M_G$ here to emphasize the role of $\lambda$ as a variable. For $M_G(\lambda)\in\W^{n\times n}$, we define the spectrum, or multiset of eigenvalues to be $$\sigma(M_G(\lambda))=\{\lambda\in\C:\det(M_G(\lambda)-\lambda I)=0\}.$$ 
Notice that $\sigma(M_G(\lambda))$ can have more than $n$ elements, some of which can be the same. 

Throughout the rest of the paper, the spectrum is understood to be a set that includes multiplicities. The element $\alpha$ of the multiset $A$ has multiplicity $m$ if there are $m$ elements of $A$ equal to $\alpha$. If $\alpha\in A$ with multiplicity $m$ and $\alpha\in B$ with multiplicity $n$, then

(i) the union $A\cup B$ is a multiset in which $\alpha$ has multiplicity $m+n$; and

(ii) the difference $A-B$ is a multiset in which $\alpha$ has multiplicity $m-n$ if $m-n>0$ and where $\alpha\not\in A-B$ otherwise.

Similarly, the multiset $A\subset B$ means for any $\alpha\in A$, we have $\alpha\in B$, and the multiplicity of $\alpha$ in $A$, is less than or equal to the mutliplicity of $\alpha$ in $B$.

An eigenvector for eigenvalue $\lambda_0\in\sigma(M_G(\lambda))$ is defined to be $u\in\C^n,u\neq0$ such that 
$$M_G(\lambda_0)u=\lambda_0u.$$ 
One can see the eigenvectors of $M_G(\lambda)\in\W^{n\times n}$ for $\lambda_0$ are the same as the eigenvectors of $M_G(\lambda_0)\in\C^{n\times n}$ for $\lambda_0$. Similarly the generalized eigenvectors of $M_G(\lambda)$ for $\lambda_0$ are the generalized eigenvectors of $M_G(\lambda_0)$ for $\lambda_0$.

A path $\gamma=(i_0,\dots,i_p)$ in the graph $G=(V,E,w)$ is an ordered sequence of distinct vertices $i_0,\dots,i_p\in V$ such that $(i_l,i_{l+1})\in E$ for $0\le l\le p-1$. The vertices $i_1,\dots,i_{p-1}\in V$ of $\gamma$ are called \emph{interior vertices}. If $i_0=i_p$ then $\gamma$ is a \emph{cycle}. A cycle is called a \emph{loop} if $p=1$ and $i_0=i_1$. The length of a path $\gamma=(i_0,\dots,i_p)$ is the integer $p$. Note that there are no paths of length 0 and that every edge $(i,j)\in E$ is a path of length 1.

If $S\subset V$ is a subset of all the vertices, we will write $\overline S=V\setminus S$ and denote by $|S|$ the cardinality of the set $S$.

\begin{df}
\emph{(structural set).} Let $G=(V,E,w)\in\G$. A nonempty vertex set $S\subset V$ is a structural set of $G$ if 
\begin{itemize}
\item each cycle of $G$, that is not a loop, contains a vertex in $S$;
\item $w(i,i)\neq\lambda$ for each $i\in\overline S$.
\end{itemize}

$S$ is called a $\lambda_0-$structural set if a structural set $S$ also satisfies $w(i,i)\neq\lambda_0,\forall i\in\overline S$ for some $\lambda_0\in\C$.
\end{df}

\begin{df}
Given a structural set $S$, a \emph{branch} of $(G,S)$ is a path $\beta=(i_0,i_1,\dots,i_{p-1},i_p)$ such that  $i_0,i_p\in V$ and all $i_1,\dots,i_{p-1}\in\overline S$.
\end{df}
We denote by $\mathcal{B}=\mathcal{B}_{G,S}$ the set of all branches of $(G,S)$.
Given vertices $i,j\in V$, we denote by $\mathcal{B}_{i,j}$ the set of all branches in $\mathcal{B}$ that start in $i$ and end in $j$. For each branch $\beta=(i_0,i_1,\dots,i_{p-1},i_p)$ we define the \emph{weight} of $\beta$ as follows:
\begin{equation}
w(\beta,\lambda):=w(i_0,i_1)\prod_{l=1}^{p-1}\frac{w(i_l,i_{l+1})}{\lambda-w(i_l,i_l)}.
\end{equation}

Given $i,j\in V$ set
\begin{equation}
R_{i,j}(G,S,\lambda):=\sum_{\beta\in\mathcal{B}_{i,j}}w(\beta,\lambda).
\end{equation}

\begin{df}
\emph{(Isospectral reduction).} Given $G\in\G$ and a structural set $S$, the reduced adjacency matrix  $R_S(G,\lambda)$ is the $|S|\times |S|-$matrix with the entries $R_{i,j}(G,S,\lambda),i,j\in S$. This adjacency matrix $R_S(G,\lambda)$ on $S$ defines the reduced graph which is the isospectral reduction of the original graph $G$.
\end{df}


\section{Generalized eigenvectors of isospectral graph reductions}

Let $\lambda_0$ be an eigenvalue of $M_G(\lambda)$ with multiplicity at least 2, and let $u=(u_1,u_2,\dots,u_n)\in\C^n$ be the corresponding eigenvector, i.e. $M_G(\lambda_0)u=\lambda_0u$. Let $v=(v_1,v_2,\dots,v_n)\in\C^n$ be the corresponding rank 2 generalized eigenvector, i.e. $M_G(\lambda_0)v-\lambda_0v=u$. Without any loss of generality we may assume that $S=\{m+1,\dots,n\}$ is a $\lambda_0-$structural set . 
It is known that $\lambda_0$ is also an eigenvalue of $R_S(G,\lambda)$, i.e. $R_S(G,\lambda_0)u_S=\lambda_0u_S$, where $u_S=(u_{m+1},\dots,u_n)$ is the restriction of $u$ to $S$. We will refer to this property from now on as the preservation of eigenvectors.
Our goal in this section is to see what happens to
generalized eigenvectors under isospectral transformations.
\begin{theorem}\label{entry-wise}
Let $S$ be a $\lambda_0-$structural set of a graph $G=(V,E,w)$. $M_G(\lambda)$ is the adjacency matrix of $G$. $u,v\in\C^n$ are the eigenvector and generalized eigenvector for $M_G(\lambda)$ such that $M_G(\lambda_0)u=\lambda_0u,M_G(\lambda_0)v-\lambda_0v=u$. Then if there is a $c\in\C$ such that $c\neq-1$ and
\begin{equation}\label{gen-eig}
\sum_{l\in\overline S}\frac{R_{il}(\lambda_0)}{\lambda_0-\omega(l,l)}u_l=cu_i,\forall i\in S,
\end{equation}
then $R_S(G,\lambda_0)v_S-\lambda_0v_S=(1+c)u_S$.
\end{theorem}
We first introduce some useful notations before proceeding to the proof of theorem \ref{entry-wise}.

Given vertices $i,j\in V$, we denote by $\mathcal{B}_{i,j}^{(p)}$ the set of all branches in $\mathcal{B}$ of length $p$ that start at $i$ and end at $j$. For any $i,j\in V$ set
$$R_{i,j}^{(p)}(G,S,\lambda):=\sum_{\beta\in\mathcal{B}_{i,j}^{(p)}}w(\beta,\lambda).$$
Therefore the  reduced weights $R_{i,j}(G,S,\lambda)$ for $S=\{m+1,\dots,n\}$ satisfy \begin{align*}R_{i,j}(G,S,\lambda)=\sum_{p=1}^{m+1}R_{i,j}^{(p)}(G,S,\lambda),\forall i,j\in S.\\
R_{i,j}(G,S,\lambda)=\sum_{p=1}^{m-1}R_{i,j}^{(p)}(G,S,\lambda),\forall i,j\in\overline{S},i\neq j.\\
R_{i,i}(G,S,\lambda)=w(i,i),\forall i\in\overline{S}.\\
R_{i,j}(G,S,\lambda)=\sum_{p=1}^mR_{i,j}^{(p)}(G,S,\lambda),\forall i\in S, j\in\overline{S}\text{ or }i\in\overline{S}, j\in S.
\end{align*}
To simplify notations we will write $R_{i,j}$ and $R_{i,j}^{(p)}$ instead of $R_{i,j}(G,S,\lambda_0)$ and $R_{i,j}^{(p)}(G,S,\lambda_0)$, respectively.
\begin{proof}
Write $v=(v_{\overline S},v_S)$, where $v_{\overline S}=(v_l)_{l\in\overline S}$ and $v_S=(v_i)_{i \in S}$. Since $M_G(\lambda_0)v=\lambda_0v+u$, we have for all $l\in\overline S$, (for convenience, all $w(i,j)$ mean $w(i,j)(\lambda_0)$ in the proof)

$$\sum_{k\in S}\omega(l,k)v_k+\omega(l,l)v_l+\sum_{l_1\in\overline S,l_1\neq l}\omega(l,l_1)v_{l_1}=\lambda_0v_l+u_l.$$
Therefore,
\begin{equation}\label{vl}
v_l=\sum_{k\in S}\frac{\omega(l,k)}{\lambda_0-\omega(l,l)}v_k+\sum_{l_1\in\overline S,l_1\neq l}\frac{\omega(l,l_1)}{\lambda_0-\omega(l,l)}v_{l_1}-\frac{u_l}{\lambda_0-\omega(l,l)}.\end{equation}
Analogously for all $i\in S$,
$$v_i=\sum_{k\in S,k\neq i}\frac{\omega(i,k)}{\lambda_0-\omega(i,i)}v_k+\sum_{l\in\overline S}\frac{\omega(i,l)}{\lambda_0-\omega(i,i)}v_l-\frac{u_i}{\lambda_0-\omega(i,i)}.$$
Substituting $v_l$'s above by \eqref{vl} gives,
\begin{align*}
&v_i=\sum_{k\in S,k\neq i}\frac{R^{(1)}_{ik}}{\lambda_0-\omega(i,i)}v_k+\sum_{k\in S,l\in\overline S}\frac{\omega(i,l)\omega(l,k)}{[\lambda_0-\omega(i,i)][\lambda_0-\omega(l,l)]}v_k\\
&+\sum_{l_1,l\in\overline S,l_1\neq l}\frac{\omega(i,l)\omega(l,l_1)}{[\lambda_0-\omega(i,i)][\lambda_0-\omega(l,l)]}v_{l_1}-\sum_{l\in\overline S}\frac{\omega(i,l)}{[\lambda_0-\omega(i,i)][\lambda_0-\omega(l,l)]}u_l\\
&-\frac{u_i}{\lambda_0-\omega(i,i)}\\
&=\sum_{k\in S,k\neq i}\frac{R_{ik}^{(1)}}{\lambda_0-\omega(i,i)}v_k+\sum_{k\in S}\frac{R_{ik}^{(2)}}{\lambda_0-\omega(i,i)}v_k+\sum_{l_1,l\in\overline S}\frac{\omega(i,l)\omega(l,l_1)}{[\lambda_0-\omega(i,i)][\lambda_0-\omega(l,l)]}v_{l_1}\\
&-\sum_{l\in\overline S}\frac{\omega(i,l)}{[\lambda_0-\omega(i,i)][\lambda_0-\omega(l,l)]}u_l-\frac{u_i}{\lambda_0-\omega(i,i)}.
\end{align*}
Proceeding inductively, we get
\begin{align*}
&v_i=\sum_{k\in S,k\neq i}\frac{R_{ik}^{(1)}}{\lambda_0-\omega(i,i)}v_k+\sum_{k\in S}\frac{R^{(2)}_{ik}}{\lambda_0-\omega(i,i)}v_k+\dots+\sum_{k\in S}\frac{R_{ik}^{(p)}}{\lambda_0-\omega(i,i)}v_k\\
&+\sum_{l_1,\dots,l_{p-1},l\in\overline S,l_r\neq l_s,l_s\neq l}\frac{\omega(i,l)\omega(l,l_1)\omega(l_1,l_2)\dots\omega(l_{p-2},l_{p-1})}{[\lambda_0-\omega(i,i)][\lambda_0-\omega(l,l)][\lambda_0-\omega(l_1,l_1)]\dots[\lambda_0-
\omega(l_{p-2},l_{p-2})]}v_{l_{p-1}}\\
&-\sum_{l_{p-2}\in\overline S}\frac{R_{il_{p-2}}^{(p-1)}}{[\lambda_0-\omega(i,i)][\lambda_0-\omega(l_{p-2},l_{p-2})]}u_{l_{p-2}}-\sum_{l_{p-3}\in\overline S}\frac{R_{il_{p-3}}^{(p-2)}}{[\lambda_0-\omega(i,i)][\lambda_0-\omega(l_{p-3},l_{p-3})]}u_{l_{p-3}}-\dots\\
&-\sum_{l\in\overline S}\frac{R_{il}^{(1)}}{[\lambda_0-\omega(i,i)][\lambda_0-\omega(l,l)]}u_l-\frac{u_i}{\lambda_0-\omega(i,i)}.
\end{align*}
The indices in the sums above which are in $\overline S$ are all distinct; because there are no non-loop cycles in $\overline S$. Since $\overline S$ has $m$ elements, after $m+1$ steps we obtain the relation
\begin{align*}
&v_i=\sum_{k\in S,k\neq i}\frac{R_{ik}^{(1)}}{\lambda_0-\omega(i,i)}v_k+\sum_{k\in S}\frac{R_{ik}^{(2)}}{\lambda_0-\omega(i,i)}v_k+\dots+\sum_{k\in S}\frac{R_{ik}^{(m+1)}}{\lambda_0-\omega(i,i)}v_k\\
&-\sum_{l_{m-1}\in\overline S}\frac{R_{il_{m-1}}^{(m)}}{[\lambda_0-\omega(i,i)][\lambda_0-\omega(l_{m-1},l_{m-1})]}u_{l_{m-1}}-\sum_{l_{m-2}\in\overline S}\frac{R_{il_{m-2}}^{(m-1)}}{[\lambda_0-\omega(i,i)][\lambda_0-\omega(l_{m-2},l_{m-2})]}u_{l_{m-2}}\\
&-\dots-\sum_{l\in\overline S}\frac{R_{il}^{(1)}}{[\lambda_0-\omega(i,i)][\lambda_0-\omega(l,l)]}u_l-\frac{u_i}{\lambda_0-\omega(i,i)}.
\end{align*}
Therefore,
$$[\lambda_0-\omega(i,i)]v_i+\sum_{l\in\overline S}\frac{R_{il}}{\lambda_0-\omega(l,l)}u_l+u_i=\sum_{k\in S,k\neq i}R_{ik}v_k+\sum_{p=2}^{m+1}R_{ii}^{(p)}v_i,$$
$$\lambda_0v_i+u_i+\sum_{l\in\overline S}\frac{R_{il}}{\lambda_0-\omega(l,l)}u_l=\sum_{k\in S}R_{ik}v_k.$$

And finally,

$$\sum_{k\in S}R_{ik}v_k-\lambda_0v_i=(1+c)u_i,\forall i \in S,$$

which implies

$$R_S(G,\lambda_0)v_S=\lambda_0v_S+(1+c)u_S.$$
\end{proof}

We say that the generalized eigenvector $v$ is preserved if relation \eqref{gen-eig} holds. Indeed it is easy to see in this case that the projection of the generalized eigenvector to $S$ is a generalized eigenvector for the reduced adjacency matrix.
\begin{rk}
Observe that we didn't use anywhere in this proof the fact that $u$ is an eigenvector. Therefore the same proof is readily applicable to generalized eigenvectors of higher ranks. One just needs to use the rank $k$ generalized eigenvector in place of $u$ and the rank $k+1$ generalized eigenvector in place of $v$.
\end{rk}
\begin{rk}
The proof is very similar to the one given in \cite{torres14}. However, we allow the weights of the original graph to take rational functions. One can check the requirement in \cite{torres14} for the weights to be complex numbers before reduction is not necessary for the proof to work. Also the weights of the reduced graph are rational functions instead of complex numbers only. The result in \cite{torres14} would only apply to the 1st reduction, even though the preservation of eigenvectors carries through a sequence of reductions (this will be further discussed in the next section).
\end{rk}

Clearly the complement to a single vertex is a structural set of a network (graph). The following statement demonstrates that by isospectrally removing a single element (vertex) of a network (graph) one gets a much simpler condition than in Theorem \ref{entry-wise}.

\begin{theorem}\label{1p}
Let $G=(V,E,w)\in\G$ be a graph with $n$ nodes and with adjacency matrix $M_G(\lambda)$. Let $\lambda_0\in\sigma(M_G(\lambda))$ be a repeated eigenvalue and let $S\subset V$ be a $\lambda_0-$structural set which has $n-1$ nodes. Suppose $\overline S=V\setminus S=\{j\}$. Then the generalized eigenvector is preserved iff $\omega(i,j)=cu_i,\forall i\in S$ for some $c\in\C$ where $u$ is an eigenvector of $M_G(\lambda)$ for eigenvalue $\lambda_0$.
\end{theorem}
\begin{proof}
We have in this case
\begin{equation}\label{1node}
\sum_{l\in\overline S}\frac{R_{il}}{\lambda_0-\omega(l,l)}u_l=\frac{R_{ij}}{\lambda_0-\omega(j,j)}u_j=c_1u_i
\end{equation}
Since $\overline S=\{j\}$, $R_{ij}=\omega(i,j)$, the relation \eqref{1node} is equivalent to $\omega(i,j)=cu_i,\forall i\in S$.
\end{proof}
\begin{cor}
Let $\overline S=\{1,\dots,m\}$ be such that the weighted graph induced by $G$ on $\overline S$ is totally disconnected, i.e. there are no edges between vertexes in $\overline S$. Then $R_{il}=\omega(i,l),\forall i\in S,l\in\overline S$ and condition \eqref{gen-eig} becomes 
\begin{equation}\label{discon}
\sum_{l\in\overline S}\frac{\omega(i,l)}{\lambda_0-\omega(l,l)}u_l=cu_i,\forall i\in S.
\end{equation}
Hence in this case the generalized eigenvector is preserved iff \eqref{discon} is true.
\end{cor}


\section{Block Matrix Approach}

The proof of Theorem \ref{entry-wise} is an entry by entry computation based on the isospectral graph reduction. Here we will use block matrices and look at the problem from the perspective of the isospectral matrix reduction, which is more general than the isospectral graph reduction because it has fewer requirements \cite{webb14}.

For any matrix $M\in\W^{n\times n}$, and any partition $S\cup\overline S=\{1,2,\dots,n\},S\cap\overline S=\varnothing$, by permutation or renaming the nodes, we can always write the matrix as $M=\begin{pmatrix}M_{\overline S\overline S} & M_{\overline SS}\\M_{S\overline S} & M_{SS}\end{pmatrix}$. The isospectral matrix reduction of $M$ onto $S$ is defined as $$R_S=M_{SS}-M_{S\overline S}(M_{\overline S\overline S}-\lambda I)^{-1}M_{\overline SS}.$$
The only requirement for $S$ here is that the inverse matrix $(M_{\overline S\overline S}-\lambda I)^{-1}$ exists. This is a more general condition than that of the isospectral graph reduction. Indeed for the isospectral graph reduction, there must be no non-loop cycles in $\overline S$, which means that under permutation $M_{\overline S\overline S}$ is a triangular matrix. Also, the weights of loops in $\overline S$ are not equal to $\lambda$. This ensures $M_{\overline S\overline S}-\lambda I$ is invertible, but it's a stronger condition.

However, when both of these conditions hold the isospectral matrix reduction gives the same reduced matrix as the isospectral graph reduction (theorem 2.1 \cite{webb14}).

We will show now that the preservation of eigenvectors directly follows from the definition of isospectral matrix reduction. 

Suppose $u$ is an eigenvector such that $M(\lambda_0)u=\lambda_0 u,\lambda_0\in\sigma(M(\lambda))$. Write $u=\begin{pmatrix}u_{\overline S}\\u_S\end{pmatrix}$. Then we have
$$M(\lambda_0)u=\begin{pmatrix}M_{\overline S\overline S}(\lambda_0) & M_{\overline SS}(\lambda_0)\\M_{S\overline S}(\lambda_0) & M_{SS}(\lambda_0)\end{pmatrix}\begin{pmatrix}u_{\overline S}\\u_S\end{pmatrix}=\lambda_0\begin{pmatrix}u_{\overline S}\\u_S\end{pmatrix}.$$
An easy computation shows that $$(M_{\overline S\overline S}(\lambda_0)-\lambda_0 I)u_{\overline S}+M_{\overline SS}(\lambda_0)u_S=0,$$
$$M_{S\overline S}(\lambda_0)u_{\overline S}+(M_{SS}(\lambda_0)-\lambda_0 I)u_S=0.$$
Assume that the matrix $(M_{\overline S\overline S}(\lambda_0)-\lambda_0 I)$ is invertible. Then the first row gives 
\begin{equation}\label{us}
u_{\overline S}=-(M_{\overline S\overline S}(\lambda_0)-\lambda_0 I)^{-1}M_{\overline SS}(\lambda_0)u_S.
\end{equation}
By plugging this relation into the second row, we get
$$-M_{S\overline S}(\lambda_0)(M_{\overline S\overline S}(\lambda_0)-\lambda_0 I)^{-1}M_{\overline SS}(\lambda_0)u_S+M_{SS}(\lambda_0)u_S=\lambda_0 u_S.$$

Observe that the left side of this equation is $R_S(\lambda_0)u_S$, where $R_S(\lambda_0)$ is the isospectral matrix reduction evaluated at $\lambda_0$. Therefore $R(\lambda_0)u_S=\lambda_0 u_S$, i.e. projections of eigenvectors of the original (adjacency) matrix (of a network) are indeed eigenvectors with the same eigenvalues of the isospectrally reduced (adjacency) matrix. Thus, the property of eigenvector preservation for isospectral reductions is proved.

This is a much shorter proof than the one  in \cite{torres14}. Moreover, it clarifies a general structure of the isospectral reduction procedure .

Now let us turn to generalized eigenvectors. In addition to $M(\lambda_0)u=\lambda_0 u$, we have $(M(\lambda_0)-\lambda_0 I)v=u$, i.e.
$$\begin{pmatrix}M_{\overline S\overline S}(\lambda_0)-\lambda_0 I & M_{\overline SS}(\lambda_0)\\M_{S\overline S}(\lambda_0) & M_{SS}(\lambda_0)-\lambda_0 I\end{pmatrix}\begin{pmatrix}v_{\overline S}\\v_S\end{pmatrix}=\begin{pmatrix}u_{\overline S}\\u_S\end{pmatrix}.$$

A simple computation gives $$(M_{\overline S\overline S}(\lambda_0)-\lambda_0 I)v_{\overline S}+M_{\overline SS}(\lambda_0)v_S=u_{\overline S},$$
$$M_{S\overline S}(\lambda_0)v_{\overline S}+(M_{SS}(\lambda_0)-\lambda_0 I)v_S=u_S.$$

Assume that the matrix $(M_{\overline S\overline S}(\lambda_0)-\lambda_0 I)$ is invertible. Then we get from the first row 
$$v_{\overline S}=(M_{\overline S\overline S}(\lambda_0)-\lambda_0 I)^{-1}u_{\overline S}-(M_{\overline S\overline S}(\lambda_0)-\lambda_0 I)^{-1}M_{\overline SS}(\lambda_0)v_S.$$

Plugging this into the second row gives 
$$M_{S\overline S}(\lambda_0)\{[M_{\overline S\overline S}(\lambda_0)-\lambda_0 I]^{-1}u_{\overline S}-[M_{\overline S\overline S}(\lambda_0)-\lambda_0 I]^{-1}M_{\overline SS}(\lambda_0)v_S\}+(M_{SS}(\lambda_0)-\lambda_0 I)v_S=u_S,$$
$$[R_S(\lambda_0)-\lambda_0 I]v_S+M_{S\overline S}(\lambda_0)[M_{\overline S\overline S}(\lambda_0)-\lambda_0 I]^{-1}u_{\overline S}=u_S.$$

It is easy to see that $v_S$ is a generalized eigenvector for $R_S(\lambda_0)$ iff
\begin{equation}\label{blockgen}
M_{S\overline S}(\lambda_0)(M_{\overline S\overline S}(\lambda_0)-\lambda_0 I)^{-1}u_{\overline S}=cu_S.
\end{equation}
One necessary condition is that $u_S$ is in the column space of $M_{S\overline S}(\lambda_0)$. In the case when $\overline S$ has only one single node, equation \eqref{blockgen} agrees with Theorem \ref{1p}, and the necessary condition that $u_S$ is in the column space of $M_{S\overline S}(\lambda_0)$ is also sufficient.

Observe that  we have not used the relation between $u_{\overline S}$ and $u_S$. Therefore generalized eigenvectors of higher ranks are preserved iff they also satisfy \eqref{blockgen}, with $u$ being a generalized eigenvector instead of the eigenvector.

On the other hand, by plugging in \eqref{us}, the relation between $u_S$ and $u_{\overline S}$, we get $$M_{S\overline S}(\lambda_0)(M_{\overline S\overline S}(\lambda_0)-\lambda_0 I)^{-2}M_{\overline SS}(\lambda_0)u_S=-cu_S.$$
Therefore the generalized eigenvector $v$ is preserved iff $u_S$ is an eigenvector of $M_{S\overline S}(\lambda_0)(M_{\overline S\overline S}(\lambda_0)-\lambda_0 I)^{-2}M_{\overline SS}(\lambda_0)$.
\begin{theorem}\label{goback}
All eigenvectors of the reduced matrix $R_S(\lambda)$ are restrictions of the eigenvectors of the original matrix $M$. The projection of eigenvectors of $M$ onto the eigenvectors of $R_S(\lambda)$ corresponding to the same eigenvalue is a bijection.
\end{theorem}
\begin{proof}
Suppose $R(\lambda_0)u=\lambda_0 u$. Then 
$$\{M_{SS}(\lambda_0)-M_{S\overline{S}}(\lambda_0)[M_{\overline{S}\overline{S}}(\lambda_0)-\lambda_0 I]^{-1}M_{\overline SS}(\lambda_0)\}u=\lambda_0 u.$$
Let $v=-[M_{\overline S\overline S}(\lambda_0)-\lambda_0 I]^{-1}M_{\overline SS}(\lambda_0)u$. Then we have
\begin{align*}
M(\lambda_0)\begin{pmatrix}v\\u\end{pmatrix}=\begin{pmatrix}
M_{\overline S\overline S}(\lambda_0) & M_{\overline SS}(\lambda_0)\\
M_{S\overline S}(\lambda_0) & M_{SS}(\lambda_0)
\end{pmatrix}\begin{pmatrix}v\\u\end{pmatrix}=\lambda_0\begin{pmatrix}v\\u\end{pmatrix}.
\end{align*}
which proves that the projection is surjective.

Suppose now that $(v,u)^T$ and $(u_{\overline S},u)^T$ are both eigenvectors of $M$ for eigenvalue $\lambda_0$. Then by \eqref{us} we have $v=u_{\overline S}=-(M_{\overline S\overline S}-\lambda_0 I)^{-1}M_{\overline SS}u$. So the projection is also injective.
\end{proof}

Proof of the following statement can be found in \cite{webb14} (corollary 1.1).

\begin{lemma}\label{algebraic} 
For a matrix $M\in\C^{n\times n}$, let $R$ be the isospectral reduction of $M$ with respect to a structural set $S\subset\{1,\dots, n\}$. Then $\sigma(R)=\sigma(M)-\sigma(M_{\overline S\overline S})$.
\end{lemma}
Hence for a given eigenvalue $\lambda_0\in\sigma(M)$, if $\lambda_0\not\in\sigma(M_{\overline S\overline S})$, then the algebraic multiplicity of $\lambda_0$ as an eigenvalue won't change after isospectral reduction of {M} to {S}. Therefore if we reduce over a $\lambda_0-$structural set, then the algebraic multiplicity of $\lambda_0$ will be preserved.

\begin{theorem}\label{mlt}
For a matrix $M\in\C^{n\times n}$, isospectral reductions preserve both the algebraic and the geometric multiplicities of any eigenvalue.
\end{theorem}

\begin{proof}
Let $\lambda_0$ be an eigenvalue of the reduced matrix $R(\lambda)$. Lemma \ref{algebraic} ensures that if we pick a $\lambda_0-$structural set the algebraic multiplicity of $\lambda_0$ will be preserved. In fact, lemma \ref{algebraic} is true as long as the reduction exists at $\lambda_0$ \cite{webb14}, i.e. the matrix $M_{\overline S\overline S}-\lambda_0I$ is invertible.

Because of the bijection between the eigenvectors before and after isospectral reduction, the geometric multiplicity of an engenvalue is also preserved.
\end{proof}

Note though that Theorem \ref{mlt} gives no information about the generalized eigenvectors. Unlike the bijective projection we have for eigenvectors, there are situations when the reduced matrix doesn't have a generalized eigenvector for the eigenvalue $\lambda_0$ although the original (nonreduced) matrix does.

The projection of the generalized eigenvector of the original matrix to its components corresponding to vertices contained in the structural set $S$ is a generalized eigenvector for the reduced matrix if and only if \eqref{blockgen} holds.

\begin{rk}
Observe that the proof of the existence of bijection between the eigenvectors, i.e. Theorem \ref{goback}, doesn't require $M$ to have complex entries. In particular, it means that $M$ could be a matrix with entries which are rational functions of $\lambda$, which are used in isospectral reductions of networks \cite{webb14}. Consequently, the geometric multiplicity of a specific eigenvalue is preserved throughout the entire sequence of isospectral reductions.

Moreover, if the original matrix has entries which are complex numbers, then the algebraic multiplicity of a specific eigenvalue is also preserved throughout the entire sequence of isospectral reductions. By the uniqueness of sequential reductions \cite{webb14} the isospectral reduction to a specific structural set is the same as the one which results in reduction to the same set $S$ via several consecutive isospectral reductions. The algebraic multiplicity of eigenvalue $\lambda_0$ at each step is equal to the algebraic multiplicity of $\lambda_0$ for the original matrix.
\end{rk}


\section{An Example of Isospectral Reductions over Different Structural sets}
The results obtained in the previous sections demonstrate that a generalized eigenvector may or may not be preserved under isospectral reductions of matrices and networks.  
In this section we consider isospectral reduction of the simple small network depicted in figure \ref{4node}. This will illustrates the different possibilities which arise after picking different structural sets. The details of the corresponding computations are presented in the Appendix. 

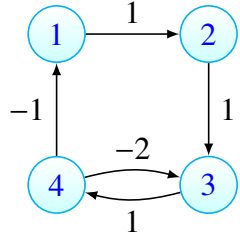
\begin {figure}[h]
\begin {tikzpicture}[-latex ,auto ,node distance =2 cm and 2cm ,on grid ,
semithick ,
state/.style ={ circle ,top color =white , bottom color = processblue!20 ,
draw,processblue , text=blue , minimum width =0.5 cm}]
\node[state] (A)
{$1$};
\node[state] (B) [right=of A] {$2$};
\node[state] (C) [below =of B] {$3$};
\node[state] (D) [below =of A] {$4$};
\path (A) edge [] node[above] {$1$} (B);
\path (B) edge [] node[right] {$1$} (C);
\path (C) edge [bend left =15] node[below] {$1$} (D);
\path (D) edge [bend left =15] node[above] {$-2$} (C);
\path (D) edge [] node[left]{$-1$} (A);
\end{tikzpicture}
\caption{Original network}\label{4node}
\end{figure}

The adjacency matrix of this network is $A=\begin{pmatrix}0&1&0&0\\0&0&1&0\\0&0&0&1\\-1&0&-2&0\end{pmatrix}$. It has eigenvalues $\{i,i,-i,-i\}$. The generalized eigenvector chain for the eigenvalue $i$ is $v^i=\begin{pmatrix}-3\\-2i\\1\\0\end{pmatrix}\to u^i=\begin{pmatrix}i\\-1\\-i\\1\end{pmatrix}$; for the eigenvalue  $-i$ the corresponding chain is $v^{-i}=\begin{pmatrix}2i\\1\\0\\1\end{pmatrix}\to u^{-i}=\begin{pmatrix}-1\\i\\1\\-i\end{pmatrix}$.

This network contains two cycles  $(1234)$ and $(34)$. All the  structural sets of size two for this network  are $S=\{1,4\},\{2,4\},\{3,4\},\{1,3\},\{2,3\}$. The list of all size 3 structural sets is $S=\{1,2,4\},\{1,3,4\},\{2,3,4\},\{1,2,3\}$.

For the size 3 structural sets, since $\overline S$ has only a single node, Theorem \ref{1p} is applicable.

If $S=\{1,2,4\},\overline S=\{3\}$, then $A_{S\overline S}=\begin{pmatrix}0\\1\\-2\end{pmatrix},u^i_S=\begin{pmatrix}i\\-1\\1\end{pmatrix},u^{-i}_S=\begin{pmatrix}-1\\i\\-i\end{pmatrix}$.Thus  $A_{S\overline S}\nparallel u^i_S$, and $A_{S\overline S}\nparallel u^{-i}_S$. $v^i_S,v^{-i}_S$ are not generalized eigenvectors for $R_S(\lambda)$.

To be more precise,
\begin{align*}
R_S(\lambda)=\begin{pmatrix} 0 & 1 & 0 \\ 0 & 0 & 1/\lambda \\ -1 & 0 & -2/\lambda \end{pmatrix},\det(R_S(\lambda)-\lambda I)=-\frac{(\lambda^2+1)^2}{\lambda}, \sigma(R_S(\lambda))=\{i,i,-i,-i\}.\\
R_S(i)=\begin{pmatrix} 0 & 1 & 0 \\ 0 & 0 & -i \\ -1 & 0 & 2i \end{pmatrix},\det(R_S(i)-\lambda I)=-(\lambda-i)(\lambda^2-i\lambda+1),\sigma(R_S(i))=\{i,\frac{1+\sqrt5}{2}i,\frac{1-\sqrt5}{2}i\}.
\end{align*}
The complex number $i$ is an eigenvalue for both $R_S(\lambda)$ and $R_S(i)$; the algebraic multiplicity of $i$ for $R_S(\lambda)$ is 2; for $R_S(i)$ it is 1. $R_S(i)$ doesn't have a generalized eigenvector for $i$. It has just one eigenvector corresponding to this eigenvalue. Therefore the generalized eigenvector is lost after isospectral reduction of the matrix.

One can check that isospectral reduction to any other size 3 structural set does not preserve the generalized eigenvectors either.

\begin {figure}
\begin {tikzpicture}[-latex ,auto ,node distance =2 cm and 2cm ,on grid ,
semithick ,
state/.style ={ circle ,top color =white , bottom color = processblue!20 ,
draw,processblue , text=blue , minimum width =0.5 cm}]
\node[draw] at (2,-1) {first reduction};
\node[state] at (0,0) (A){$1$};
\node[state] (C)[above right=of A]{$4$};
\node[state] (B) [below right=of C] {$2$};
\path (A) edge [] node[above] {$1$} (B);
\path (B) edge [] node[right] {$1/\lambda$} (C);
\path (C) edge [] node[below] {$-1$} (A);
\path (C) edge [loop above] node[above] {$-2/\lambda$} (C);
\draw[->,line width=1pt](5,1) to (7,1);
\node[state] at (8.5,1) (D){$1$};
\node[state] at (11,1) (E){$4$};
\path (D) edge [bend left] node [above] {$1/\lambda^2$} (E);
\path (E) edge [bend left] node [below] {$-1$} (D);
\path (E) edge [loop right] node [right] {$-2/\lambda$} (E);
\node[draw] at (10,-1) {second reduction};
\end{tikzpicture}
\caption{Isospectral reductions of the original network}\label{3node}
\end{figure}
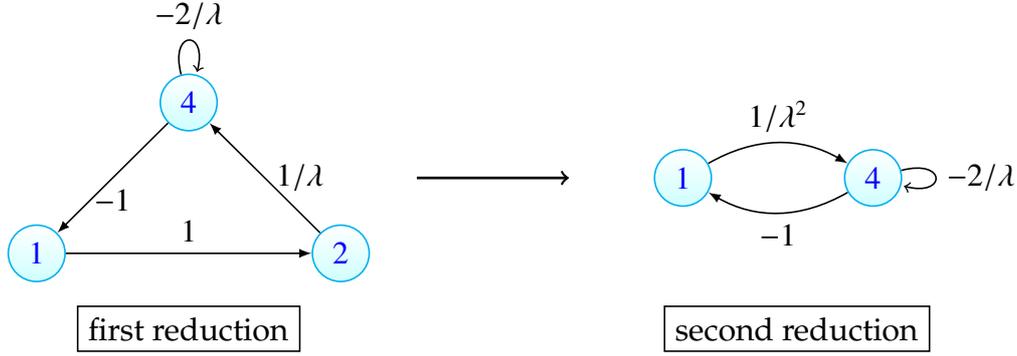

Now let us further reduce the network to $S'=\{1,4\}\subset S=\{1,2,4\}$. We  have
\begin{align*}
R_{S'}(\lambda)=\begin{pmatrix}0 & 1/\lambda^2 \\ -1 & -2/\lambda\end{pmatrix}, \det(R_{S'}(\lambda)-\lambda I)=\frac{(\lambda^2+1)^2}{\lambda^2}, \sigma(R_{S'}(\lambda))=\{i,i,-i,-i\};\\
R_{S'}(i)=\begin{pmatrix}0 & -1 \\ -1 & 2i\end{pmatrix}, \det(R_{S'}(i)-\lambda I)=(\lambda-i)^2,\sigma(R_{S'}(i))=\{i,i\};
\end{align*}

Here the algebraic multiplicity of $i$ as an eigenvalue is the same for $R_{S'}(\lambda)$ and $R_{S'}(i)$. We know that the eigenvectors are always preserved because of the bijective projection. Therefore after the second reduction we gained the generalized eigenvector back. This is a quite unexpected result which raises a question about the conditions on a structural set which allow for preservation of generalized eigenvectors.

Of course we can directly reduce the original network to $S'=\{1,4\}$ too. One can check the reduction satisfies both the entry-wise formula \eqref{gen-eig} and the block-wise formula \eqref{blockgen}. Furthermore, for all the size 2 structural sets, the reduction preserves both generalized eigenvectors ($v^i,v^{-i}$) except for the  structural set $\{3,4\}$. Observe that it is the only structural set of size two which contains a complete cycle of our network.

\begin{rk}
Let a matrix $M\in\W^{n\times n},\lambda_0\in\sigma(M)$. Define $a(\lambda_0,M)$ and $g(\lambda_0,M)$ to be the algebraic and geometric multiplicities of $\lambda_0$. Then for $M\in\C^{n\times n}$, the number of linearly independent generalized eigenvectors corresponding to $\lambda_0$ for $M$ is $d(\lambda_0,M)=a(\lambda_0,M)-g(\lambda_0,M)$.

Consider now $R(\lambda)\in\W^{n\times n}$. By definition the eigenvectors satisfy $R(\lambda_0)u=\lambda_0u$; and the generalized eigenvectors satisfy $(R(\lambda_0)-\lambda_0I)^kv=0$. Thus $g(\lambda_0,R(\lambda))=g(\lambda_0,R(\lambda_0))$, $d(\lambda_0,R(\lambda))=d(\lambda_0,R(\lambda_0))$. Observe now that $R(\lambda_0)\in\C^{n\times n}$. $d(\lambda_0,R(\lambda_0))=a(\lambda_0,R(\lambda_0))-g(\lambda_0,R(\lambda_0))$. As seen in the previous example, $a(\lambda_0,R(\lambda_0))\neq a(\lambda_0,R(\lambda))$. Hence
$$d(\lambda_0,R(\lambda))=a(\lambda_0,R(\lambda_0))-g(\lambda_0,R(\lambda_0))\neq a(\lambda_0,R(\lambda))-g(\lambda_0,R(\lambda)).$$
Therefore the number of linearly independent generalized eigenvectors is not equal to the difference between the algebraic and geometric multiplicities of the eigenvalue. In other words, the notion of generalized eigenvectors does not make much sense for matrices with rational functions as entries.
\end{rk}

\begin{rk}
Another fact worth noticing is that the reductions shown in this example form a sequence of isospectral reductions, i.e. $R_{S'}$ is an isospectral reduction of $R_S$. With the uniqueness of sequential reductions \cite{webb14}, one is tempted to believe that a property that's true for the final step of a sequence of reductions should be true in each and every step through the sequence of reductions. In our case, for the preservation of generalized eigenvectors there is no sequential property. Indeed although the generalized eigenvectors are lost in $R_S$, they managed to "come back" in $R_{S'}$. One might ask then under which conditions the generalized eigenvectors can be recovered. 

Consider the sequence of reductions that starts from $R_S$, instead of $A$. After one reduction to $R_{S'}$, instead of "recovering" generalized eigenvectors, the reduction generated generalized eigenvectors that $R_S$ doesn't have. This again, is caused by the fact that the concept of generalized eigenvectors does not actually apply to matrices whose entries are rational functions. Through a sequence of reductions, the generalized eigenvectors can be lost or "recovered", or even generated, at each step. We can not say what is going to happen in the next step in a sequence of isospectral reductions even with a knowledge of all the previous steps. Instead, one must directly analyze each new step in the sequence of reductions. 
\end{rk}


\section{Some Sufficient conditions for preservation of generalized eigenvectors}

\begin{theorem}
The generalized eigenvectors are preserved if either of the following conditions hold
: (i) $M_{S\overline S}(\lambda_0)=0$; (ii) $M_{\overline S S}(\lambda_0)=0$.
\end{theorem}
\begin{proof}
If $M_{S\overline S}(\lambda_0)=0$, plugging in \eqref{blockgen} we have $M_{S\overline S}(\lambda_0)(M_{\overline S\overline S}(\lambda_0)-\lambda_0I)^{-1}u_{\overline S}=0=0u_S$.

If $M_{\overline S S}(\lambda_0)=0$, then by \eqref{us} we have $u_{\overline S}=-(M_{\overline S\overline S}(\lambda_0)-\lambda_0I)^{-1}M_{\overline S S}(\lambda_0)u_S=0$. If $u_{\overline S}=0$, then \eqref{blockgen} is true.
\end{proof}
For a network, the relation $M_{S\overline S}(\lambda_0)=0$ means that no edges go from $S$ to $\overline S$; while $M_{\overline SS}(\lambda_0)=0$ means there is no edge from $\overline S$ to $S$. In either case, we have $R(\lambda_0)=M_{SS}(\lambda_0)$.

\section{Reconstruction of the original network}

In this section we address the problem of reconstructing the original network or matrix from its isospectral reduction.

The eigenvectors for eigenvalue $\lambda_0$ can be reconstructed, as shown in \cite{torres14}. And we can reconstruct the generalized eigenvectors for $\lambda_0$ similarly.

\begin{df}
(Depth of a vertex) The depth of a vertex $i\in V$ is defined recursively as follows.

(1) A vertex $i\in S$ has depth 0.

(2) A vertex $i\in\overline S$ has depth $k$ iff $i$ has no depth less than $k$, and $(i,j)\in E$ implies $j$ has depth $<k$, for all $j\in V$.\\
We denote by $S_k$ the set of all vertices of depth $\le k$. Because $S$ is a structural set, every vertex $i$ has a finite depth. We set the depth of $(G,S)$ to be the maximum depth of a vertex.
\end{df}
\begin{prop}\label{rec}
If $u_S=(u_i^S)_{i\in S},v_S=(v_i^S)$ are the eigenvector and rank 2 generalized eigenvector of $R_S(G,\lambda_0)$, and $R_S(G,\lambda_0)v_S-\lambda_0v_S=(1+c)u_S$, where $c\neq-1$ is a complex number, then the recursive relations
\begin{equation}\label{reconstruction}
\begin{cases}
v_i=v_i^S\qquad\text{ for }i\in S_0=S\\
\displaystyle v_l+\frac{u_l}{\lambda_0-\omega(l,l)}=\frac{1}{\lambda_0-\omega(l,l)}\sum_{j\in S_{k-1}}\omega(l,j)v_j\qquad\text{ for all }l\in S_k\setminus S_{k-1}
\end{cases}
\end{equation}
determine the rank 2 generalized eigenvector $v$ for $M_G$ associated to $\lambda_0$.
\end{prop}
\begin{rk}
The relation \eqref{reconstruction} comes from equation \eqref{vl}. And this reconstruction formula can reconstruct higher ranking generalized eigenvectors as well. We just need to replace $u$ with the rank $k-1$ generalized eigenvector and $v$ with the rank $k$ generalized eigenvector.
\end{rk}
\begin{rk}
Proposition \ref{rec} is true for any $M_G\in\W^{n\times n}$. For a matrix $M\in\C^{n\times n}$, if all its eigenvalues, eigenvectors and chains of generalized eigenvectors are preserved in an isospectral reduction, then the Jordan form of $M$ is known, and its corresponding eigenvectors and chains of generalized eigenvectors can be reconstructed. We can reconstruct the original matrix $M$. This reconstruction is unique up to permutation of the nodes by $M=BJB^{-1}$, where $J$ is the Jordan form and $B=[u_1,u_2,v_2,\dots]$ are the corresponding eigenvectors and generalized eigenvectors.
\end{rk}

\section{Spectral equivalence of networks and of complex matrices}

In this section we introduce a more general than in \cite{webb14} notion of spectral equivalence of networks and compare it with standard spectral equivalence of matrices with complex entries.

Recall that the spectrum, $\sigma$, of a matrix is the union of all eigenvalues together with their multiplicities.

Let $\W_\pi\subset\W$ be the set of rational functions $p(\lambda)/q(\lambda)$ such that $\text{deg}(p)\le\text{deg}(q)$, where $\text{deg}(p)$ is the degree of the polynomial $p(\lambda)$. And let $\G_\pi\subset\G$ be the set of graphs $G=(V,E,w)$ such that $w:E\to\W_\pi$. Every graph in $\G_\pi$ can be isospectrally reduced \cite{webb14}.

Two weighted directed graphs $G_1=(V_1,E_1,w_1)$ and $G_2=(V_2,E_2,w_2)$ are \emph{isomorphic} if there is a bijection $b:V_1\to V_2$ such that there is an edge $e_{ij}$ in $G_1$ from $v_i$ to $v_j$ if and only if there is an edge $\tilde e_{ij}$ between $b(v_i)$ and $b(v_2)$ in $G_2$ with $w_2(\tilde e_{ij})=w_1(e_{ij})$. If the map $b$ exists, it is called an \emph{isomorphism}, and we write $G_1\simeq G_2$.

An isomorphism is essentially a relabeling of the vertices of a graph. Therefore, if two graphs are isomorphic, then their spectra are identical. The relation of being isomorphic is reflexive, symmetric, and transitive; in other words, it's an equivalence relation.

\begin{df}[Generalized Spectral Equivalence of Graphs]
Suppose that for each graph $G=(V,E,w)$ in $\G_\pi$, $\tau$ is a rule that selects a unique nonempty subset $\tau(G)\subset V$. Let $R_\tau$  be the isospectral reduction of $G$ onto $\tau(G)$. Then $R_\tau$ induces an equivalence relation $\sim$ on the set $\G_\pi$, where $G\sim H$ if $R_\tau^m(G)\simeq R_\tau^k(H)$ for some $m,k\in\mathbb{N}$ .
\end{df}

\begin{rk}
Observe that we do not require $\tau(G)$ to be a structural subset of $G$. However there is a unique isospectral reduction \cite{webb14} (possibly via a sequence of isospectral reductions to structural sets if $\tau(G)$ is not a structural subset of $G$) of $G$ onto $\tau(G)$.

Our definition of spectral equivalence of networks (graphs) is more general than the one in \cite{webb14}, where it was required that $m=k=1$. Therefore our classes of spectrally equivalent networks are larger than the ones considered in \cite{webb14}. Namely each class of equivalence in our sense consists of a countable number of equivalence classes in the sense of \cite{webb14}. Our approach/definition could be of use for analysis of real world networks many of which have a hierarchical structure \cite{ACM}, \cite{JKAM}.

Clearly any nonzero number is an eigenvector of any dimension 1 matrix. For this reason we do not consider reductions to one node since at that point all the geometric properties are lost.
\end{rk}

\begin{proof}
It is easy to see that the relation defined is reflexive and symmetric.

Suppose that $G\sim H$, with $R_\tau^m(G)\simeq R_\tau^s(H)$; $H\sim K$, with $R_\tau^r(H)\simeq R_\tau^t(K)$. Without any loss of generality, we assume $r>s$. Then
$$R_\tau^{m+r-s}(G)\simeq R_\tau^r(H)\simeq R_\tau^t(K),G\sim K.$$
\end{proof}

We call matrices that can be isospectrally reduced to the same matrix (up to permutation) spectrally equivalent. By lemma \ref{algebraic}, we have $\sigma(M)=\sigma(R)\cup[\sigma(M)\cap\sigma(M_{\overline S\overline S})]$ for $M\in\C^{n\times n}$. If 

\begin{equation}\label{seq}
\sigma(M)\cap\sigma(M_{\overline S\overline S})=\varnothing,
\end{equation} we have $\sigma(M)=\sigma(R)$.

\begin{prop}
Let $M_1,M_2\in\C^{n\times n}$, both can be reduced to the same matrix $R(\lambda)\in\W^{m\times m}$. Let them both satisfy \eqref{seq}. Then $M_1$ and $M_2$ have the same eigenvalues, with the same algebraic and geometric multiplicities for each eigenvalue. They also have the same eigenvectors for each eigenvalue. However, they can still have different Jordan forms, since the generalized eigenvectors are generally not preserved by isospectral reductions.
\end{prop}
\begin{proof}
Since $M_1$ and $M_2$ both satisfy \eqref{seq}, we have $\sigma(M_1)=\sigma(R)=\sigma(M_2)$, i.e. $M_1$ and $M_2$ have the same eigenvalues and the same algebraic multiplicity for each eigenvalue.

Theorem \ref{goback} implies that the eigenvectors of $R(\lambda)$ are bijective projections of eigenvectors of $M_1$, as well as $M_2$. By the reconstruction of eigenvectors in \cite{torres14}, we know $M_1$ and $M_2$ have the same eigenvectors for each eigenvalue, thus the same geometric multiplicity for each eigenvalue.

However, two matrices with the same eigenvalues, with the same algebraic and geometric multiplicities for each eigenvalue, and the same eigenvectors for each eigenvalue can still have different Jordan form. For example,
$$A_1=\begin{pmatrix}5 & 1 & 0 & 0\\ 0 & 5 & 0 & 0 \\ 0 & 0 & 5 & 1\\ 0 & 0 & 0 & 5\end{pmatrix}=
\begin{pmatrix}1 & 0 & 0 & 0 \\ 0 & 1 & 0 & 0 \\ 0 & 0 & 1 & 0 \\ 0 & 0 & 0 & 1\end{pmatrix}\begin{pmatrix}5 & 1 & 0 & 0\\ 0 & 5 & 0 & 0 \\ 0 & 0 & 5 & 1 \\ 0 & 0 & 0 & 5\end{pmatrix}
\begin{pmatrix}1 & 0 & 0 & 0 \\ 0 & 1 & 0 & 0 \\ 0 & 0 & 1 & 0 \\ 0 & 0 & 0 & 1\end{pmatrix},$$
$$A_2=\begin{pmatrix}5 & 0 & 0 & 0 \\ 0 & 5 & 0 & 1 \\ 0 & 1 & 5 & 0\\ 0 & 0 & 0 & 5\end{pmatrix}=\begin{pmatrix}1 & 0 & 0 & 0 \\ 0 & 0 & 1 & 0 \\ 0 & 1 & 0 & 0 \\ 0 & 0 & 0 & 1\end{pmatrix}
\begin{pmatrix}5& 0 & 0 & 0 \\ 0 & 5 & 1 & 0 \\ 0 & 0 & 5 & 1 \\ 0 & 0 & 0 & 5\end{pmatrix}\begin{pmatrix}1 & 0 & 0 & 0 \\ 0 & 0 & 1 & 0 \\ 0 & 1 & 0 & 0 \\ 0 & 0 & 0 & 1\end{pmatrix},$$
$A_1$ and $A_2$ both have eigenvalue $5$ with algebraic multiplicity $4$ and geometric multiplicity $2$. They also have the same eigenvectors for eigenvalue $5$, i.e. $u_1=(1,0,0,0)^T,u_2=(0,0,1,0)^T$. But $A_1$'s Jordan form consists of 2 size 2 Jordan blocks and $A_2$'s Jordan form consists of 1 simple eigenvalue and a size 3 Jordan block, they have different Jordan forms.
\end{proof}

If $M_1$ satisfies \eqref{seq} but $M_2$ does not, it is known that $M_2$ loses some eigenvalues (those in the intersection in \eqref{seq}) when reduced to $R$ while $M_1$ does not.  Therefore, $\sigma(M_2)\supsetneq\sigma(M_1)$ and the matrix $M_2$ has a higher dimension than $M_1$. 

Not all similar matrices are spectrally equivalent. For example, a matrix that is already in Jordan form always has eigenvalues in $\overline S$. It will lose eigenvalues in $\overline S$ after reduction. For similar matrices that satisfy \eqref{seq}, their isospectral reductions will have the same eigenvalues, with the same algebraic and geometric multiplicities. However, reductions of these matrices may not be the same.

For example, matrices $A$ and $B$ down below have the same eigenvalues.
$$A=\begin{pmatrix}1&5&2\\3&6&8\\4&7&9\end{pmatrix}^{-1}\begin{pmatrix}1&0&0\\0&2&0\\0&0&3\end{pmatrix}\begin{pmatrix}1&5&2\\3&6&8\\4&7&9\end{pmatrix}
=\frac{1}{17}\begin{pmatrix}148&206&256\\-13&-5&-28\\-33&-48&-41\end{pmatrix},$$
$$B=\begin{pmatrix}2 & 1 & 0 \\ 7 & 3 & 5 \\ 8 & 9 & 4\end{pmatrix}^{-1}\begin{pmatrix}1 & 0 & 0 \\ 0 & 2 & 0 \\ 0 & 0 & 3\end{pmatrix}\begin{pmatrix}2 & 1 & 0 \\ 7 & 3 & 5 \\ 8 & 9 & 4\end{pmatrix}
=\frac{1}{27}\begin{pmatrix}1 & -39 & -10 \\ 52 & 105 & 20 \\ 43 & 24 & 56\end{pmatrix}.$$

The matrix $A$ has 3 listed below dimension-2 isospectral reductions.

$$R_{12}(A)=\frac{1}{17\lambda+41}\begin{pmatrix}148\lambda-140 & 206\lambda-226\\23-13\lambda & 67-5\lambda\end{pmatrix},$$
$$R_{13}(A)=\frac{1}{17\lambda+5}\begin{pmatrix}148\lambda-114 & 256\lambda-264\\27-33\lambda & 67-41\lambda\end{pmatrix},$$
$$R_{23}(A)=\frac{1}{17\lambda-148}\begin{pmatrix}-5\lambda-114 & 48-28\lambda\\18-48\lambda & -41\lambda-140\end{pmatrix}.$$

The matrix $B$ also has 3 dimension-2 isospectral reductions.
$$R_{12}(B)=\frac{1}{27\lambda-56}\begin{pmatrix}\lambda-18 & 72-39\lambda\\ 52\lambda-76 & 105\lambda-200\end{pmatrix},$$
$$R_{13}(B)=\frac{1}{27\lambda-105}\begin{pmatrix}\lambda-79 & 10-10\lambda \\ 43\lambda-121 & 56\lambda-200\end{pmatrix},$$
$$R_{23}(B)=\frac{1}{27\lambda-1}\begin{pmatrix}105\lambda-79 & 20\lambda-20 \\ 24\lambda-63 & 56\lambda-18\end{pmatrix}.$$

It is easy to see that there is no pair of reductions, one for $A$ and one for $B$, which are the same, meaning that one is a permutation of the other. Even though $A$ and $B$ are similar matrices that both satisfy \eqref{seq}, they are not spectrally equivalent.

When \eqref{seq} does not hold, the eigenvalues which belong to both $\sigma(M)$ and $\sigma(M_{\overline S\overline S})$ will be lost after reduction or their multiplicities will decrease. Theorems 1, 2, 3, 4 and 5 require $(M_{\overline S\overline S}(\lambda_0)-\lambda_0I)$ to be invertible, which implies $\lambda_0\not\in\sigma(M_{\overline S\overline S})$. Therefore in this case $\lambda_0$ doesn't belong to the intersection in \eqref{seq}.


\appendix

\section{Detailed computations for the network in figure 1}
\label{appendix-sec1}

We provide here the exact analytic computations  for the network depicted in Fig.1. For each structural set of this network the  entry-wise condition \eqref{gen-eig} for example 1 is verified.
\begin{itemize}

\item $S=\{1,4\}$\\
For eigenvalue $i$,
\begin{align*}
\frac{R_{12}}{i-\omega(2,2)}u_2+\frac{R_{13}}{i-\omega(3,3)}u_3=-i(R_{12}(-1)-iR_{13})=iR_{12}-R_{13}=i+i=2i=2u_1;\\
\frac{R_{42}}{i-\omega(2,2)}u_2+\frac{R_{43}}{i-\omega(3,3)}u_3=iR_{42}-R_{43}=0-(-2)=2=2u_4.
\end{align*}
For eigenvalue $-i$,
\begin{align*}
\frac{R_{12}}{-i-\omega(2,2)}u_2+\frac{R_{13}}{-i-\omega(3,3)}u_3=i(R_{12}i+R_{13})=-R_{12}+iR_{13}=-1-1=-2=2u_1;\\
\frac{R_{42}}{-i-\omega(2,2)}u_2+\frac{R_{43}}{-i-\omega(3,3)}u_3=-R_{42}+iR_{43}=0-2i=-2i=2u_4.
\end{align*}
One can check that the reduction preserves the generalized eigenvectors in this case.

\item $S=\{2,4\}$\\
For eigenvalue $i$,
\begin{align*}
\frac{R_{21}}{i-\omega(1,1)}u_1+\frac{R_{23}}{i-\omega(3,3)}u_3=-i(iR_{21}-iR_{23})=R_{21}-R_{23}=-1=u_2;\\
\frac{R_{41}}{i-\omega(1,1)}u_1+\frac{R_{43}}{i-\omega(3,3)}u_3=R_{41}-R_{43}=-1-(-2)=1=u_4.
\end{align*}
For eigenvalue $-i$,
\begin{align*}
\frac{R_{21}}{-i-\omega(1,1)}u_1+\frac{R_{23}}{-i-\omega(3,3)}u_3=i(-R_{21}+R_{23})=-iR_{21}+iR_{23}=0+i=i=u_2;\\
\frac{R_{41}}{-i-\omega(1,1)}u_1+\frac{R_{43}}{-i-\omega(3,3)}u_3=-iR_{41}+iR_{43}=i-2i=-i=u_4.
\end{align*}
One can check that the reduction preserves the generalized eigenvectors here.

\item $S=\{3,4\}$\\
For eigenvalue $i$,
\begin{align*}
\frac{R_{31}}{i-\omega(1,1)}u_1+\frac{R_{32}}{i-\omega(2,2)}u_2=-i(iR_{31}-R_{32})=R_{31}+iR_{32}=0;\\
\frac{R_{41}}{i-\omega(1,1)}u_1+\frac{R_{42}}{i-\omega(2,2)}u_2=R_{41}+iR_{42}=-1-1=-2=-2u_4.
\end{align*}
For eigenvalue $-i$,
\begin{align*}
\frac{R_{31}}{-i-\omega(1,1)}u_1+\frac{R_{32}}{-i-\omega(2,2)}u_2=i(-R_{31}+iR_{32})=-iR_{31}-R_{32}=0;\\
\frac{R_{41}}{-i-\omega(1,1)}u_1+\frac{R_{42}}{-i-\omega(2,2)}u_2=-iR_{41}-R_{42}=i+i=2i=-2u_4.
\end{align*}
Here the generalized eigenvectors are not preserved. Observe that the structural set in this case contains a complete cycle.

\item $S=\{1,3\}$\\
For eigenvalue $i$,
\begin{align*}
\frac{R_{12}}{i-\omega(2,2)}u_2+\frac{R_{14}}{i-\omega(4,4)}u_4=-i(-R_{12}+R_{14})=iR_{12}-iR_{14}=i=u_1;\\
\frac{R_{32}}{i-\omega(2,2)}u_2+\frac{R_{34}}{i-\omega(4,4)}u_4=iR_{32}-iR_{34}=-i=u_3.
\end{align*}
For eigenvalue $-i$,
\begin{align*}
\frac{R_{12}}{-i-\omega(2,2)}u_2+\frac{R_{14}}{-i-\omega(4,4)}u_4=i(iR_{12}-iR_{14})=-R_{12}+R_{14}=-1=u_1;\\
\frac{R_{32}}{-i-\omega(2,2)}u_2+\frac{R_{34}}{-i-\omega(4,4)}u_4=-R_{32}+R_{34}=1=u_3.
\end{align*}
One can check that the reduction preserves the generalized eigenvectors here.

\item $S=\{2,3\}$\\
For eigenvalue $i$,
\begin{align*}
\frac{R_{21}}{i-\omega(1,1)}u_1+\frac{R_{24}}{i-\omega(4,4)}u_4=-i(iR_{21}+R_{24})=R_{21}-iR_{24}=0;\\
\frac{R_{31}}{i-\omega(1,1)}u_1+\frac{R_{34}}{i-\omega(4,4)}u_4=R_{31}-iR_{34}=0.
\end{align*}
For eigenvalue $-i$,
\begin{align*}
\frac{R_{21}}{-i-\omega(1,1)}u_1+\frac{R_{24}}{-i-\omega(4,4)}u_4=i(-R_{12}-iR_{14})=-iR_{21}+R_{24}=0;\\
\frac{R_{31}}{-i-\omega(1,1)}u_1+\frac{R_{34}}{-i-\omega(4,4)}u_4=-iR_{31}+R_{34}=-1+1=0.
\end{align*}
One can check that the reduction preserves the generalized eigenvectors here.

\item $S=\{1,2,4\}$\\
For eigenvalue $i$,
\begin{align*}
\frac{R_{13}}{i-\omega(3,3)}u_3=-R_{13}=0;\quad\frac{R_{23}}{i-\omega(3,3)}u_3=-R_{23}=-1=u_2;\quad\frac{R_{43}}{i-\omega(3,3)}u_3=2u_4.
\end{align*}
For eigenvalue $-i$,
\begin{align*}
\frac{R_{13}}{-i-\omega(3,3)}u_3=iR_{13}=0;\quad\frac{R_{23}}{-i-\omega(3,3)}u_3=iR_{23}=i=u_2;\quad\frac{R_{43}}{-i-\omega(3,3)}u_3=2u_4.
\end{align*}
This does not satisfy the condition. One can check that the reduction only preserves the eigenvector here.

\item $S=\{1,3,4\}$\\
For eigenvalue $i$,
\begin{align*}
\frac{R_{12}}{i-\omega(2,2)}u_2=iR_{12}=i=u_1;\quad\frac{R_{32}}{i-\omega(2,2)}u_2=iR_{32}=0;\quad\frac{R_{42}}{i-\omega(2,2)}u_2=0.
\end{align*}
For eigenvalue $-i$,
\begin{align*}
\frac{R_{12}}{-i-\omega(2,2)}u_2=-R_{12}=-1=u_1;\quad\frac{R_{32}}{-i-\omega(2,2)}u_2=-R_{32}=0;\quad\frac{R_{42}}{-i-\omega(2,2)}u_2=0.
\end{align*}
This does not satisfy the condition. One can check that the reduction only preserves the eigenvector here.

\item $S=\{2,3,4\}$\\
For eigenvalue $i$,
\begin{align*}
\frac{R_{21}}{i-\omega(1,1)}u_1=R_{21}=0;\quad\frac{R_{31}}{i-\omega(1,1)}u_1=R_{31}=0;\quad\frac{R_{41}}{i-\omega(1,1)}u_1=-1=-u_4.
\end{align*}
For eigenvalue $-i$,
\begin{align*}
\frac{R_{21}}{-i-\omega(1,1)}u_1=-iR_{21}=0;\quad\frac{R_{31}}{-i-\omega(1,1)}u_1=0;\quad\frac{R_{41}}{-i-\omega(1,1)}u_1=i=-u_4.
\end{align*}
This does not satisfy the condition. One can check that the reduction only preserves the eigenvector here.

\item $S=\{1,2,3\}$\\
For eigenvalue $i$,
\begin{align*}
\frac{R_{14}}{i-\omega(4,4)}u_4=-iR_{14}=0;\quad\frac{R_{24}}{i-\omega(4,4)}u_4=0;\quad\frac{R_{34}}{i-\omega(4,4)}u_4=-i=u_3.
\end{align*}
For eigenvalue $-i$,
\begin{align*}
\frac{R_{14}}{-i-\omega(4,4)}u_4=R_{14}=0;\quad\frac{R_{24}}{-i-\omega(4,4)}u_4=0;\quad\frac{R_{34}}{-i-\omega(4,4)}u_4=1=u_3.
\end{align*}
This does not satisfy the condition. One can check that the reduction only preserves the eigenvector here.
\end{itemize}

\section*{Acknowledgements}

This work was partially supported by the NSF grant DMS-1600568.



\end{document}